\documentclass[11pt]{amsart}

\usepackage[utf8]{inputenc}
\usepackage{amsmath}
\usepackage{amsfonts}
\usepackage{amsthm}
\usepackage{mathtools}
\usepackage{enumitem}
\usepackage{mathrsfs}
\usepackage{datetime}
\usepackage[hidelinks]{hyperref}
\usepackage{bookmark}

\newtheorem{theorem}{Theorem}[section]
\newtheorem{proposition}[theorem]{Proposition}
\newtheorem{lemma}[theorem]{Lemma}
\newtheorem{corollary}[theorem]{Corollary}
\theoremstyle{definition}
\newtheorem{definition}[theorem]{Definition}
\newtheorem{example}[theorem]{Example}
\newtheorem{remark}[theorem]{Remark}

\numberwithin{equation}{section}

\DeclareEmphSequence{\bfseries\itshape}

\DeclareMathOperator{\syl}{syl}
\DeclareMathOperator{\len}{len}

\DeclareMathOperator{\tr}{tr}

\DeclareMathOperator{\re}{Re}
\DeclareMathOperator{\im}{Im}

\DeclareMathOperator{\SL}{SL}
\DeclareMathOperator{\PSL}{PSL}

\newcommand{\Q}{\mathbb{Q}}
\newcommand{\R}{\mathbb{R}}
\newcommand{\C}{\mathbb{C}}
\newcommand{\Z}{\mathbb{Z}}

\newcommand{\N}{\mathbb{N}}
\newcommand{\HH}{\mathbb{H}}

\newcommand\restr[2]{{
  \left.\kern-\nulldelimiterspace
  #1
  \right|_{#2}}}

\begin{document}

\title{Long relators in groups generated by two parabolic elements}
\author{Rotem Yaari}
% \date{\today}
\address{Rotem Yaari, Department of Pure Mathematics, Tel Aviv University, Israel}
\email{rotemnaory@mail.tau.ac.il}

\begin{abstract}
We find a family of groups generated by a pair of parabolic elements in which every relator must admit a long subword of a specific form.
In particular, this collection contains groups in which the number of syllables of any relator is arbitrarily large.
This suggests that the existing methods for finding non-free groups with rational parabolic generators may be inadequate in this case, as they depend on the presence of relators with few syllables.
Our results rely on two variants of the ping-pong lemma that we develop, applicable to groups that are possibly non-free. These variants aim to isolate the group elements responsible for the failure of the classical ping-pong lemma.
\end{abstract}

% \keywords{a}

% \subjclass{1}

\maketitle

\section{Introduction}
\subsection{Background and main results}
Determining which pairs of parabolic elements of $\SL_2(\C)$ (namely, elements with trace $\pm2$) generate a free subgroup is a long-standing problem. For $\lambda,\mu\in\C$, denote throughout this paper
\begin{equation}\label{matrix-notation}
    A_{\lambda} = \begin{pmatrix}
    1 &  \lambda \\
    0 & 1
    \end{pmatrix},\quad B_{\mu} = \begin{pmatrix}
    1 & 0 \\
    \mu & 1
    \end{pmatrix}\qquad \text{and}\qquad G_{\lambda,\mu} = \langle A_{\lambda},B_{\mu}\rangle,\quad G_{\lambda} = \langle A_{\lambda},B_{\lambda}\rangle = G_{\lambda,\lambda}.
\end{equation}
Any two non-commuting parabolic matrices are simultaneously conjugate to $\pm A_\lambda$ and $\pm B_\lambda$ for some $\lambda\ne 0$ (see for example, \cite[Appendix III]{harpe-coxeter}), and hence the problem reduces to the study of $G_\lambda$ (and we also note that $G_\lambda$ is free if and only if its image in $\PSL_2(\C)$ is free).
However, it will still be convenient later to replace the group $G_\lambda$ with a conjugate group, using the fact that when $\lambda_1\mu_1=\lambda_2\mu_2$, $G_{\lambda_1,\mu_1}$ and $G_{\lambda_2,\mu_2}$ are conjugate \cite[Appendix III]{harpe-coxeter}. 
Let us call $\lambda\in\C$ \emph{free} if $G_\lambda$ is free, and \emph{non-free} otherwise.
Every transcendental number is free \cite{fuchs}, and Sanov famously proved that $2$ is free, while also characterizing the elements of $G_2$ \cite{sanov}.
Sanov's result was later generalized by Brenner, who showed that every $\lambda \ge 2$ is free \cite{brenner55}.
In \cite{chang-ree, lyndon, ignatov}, the well known ping-pong lemma was used to find large domains of free complex numbers.

Non-free numbers are abundant as well: they are dense in the unit ball in $\C$ \cite{lyndon} and real non-free numbers are dense in the interval $(-2,2)$ \cite{chang-ree, beardon}.
Among non-free numbers, the rational ones are of particular interest: throughout the years, many non-free rationals contained in $\Q\cap(-2,2)$ have been found, whereas we are not aware of a single free number in this subset. The question of whether all $\lambda \in \Q\cap(-2,2)$ are non-free was raised by Lyndon and Ullman \cite{lyndon} (as well as by Merzlyakov \cite[Problem 15.83]{unsolved_problems}), while it is explicitly stated as a conjecture in \cite{kim-koberda}.
A key observation made by Lyndon and Ullman, in an effort to find such rationals, is the following:
consider the word $w= A_\lambda^{n_1} B_\lambda^{n_2}\cdots A_\lambda^{n_k}$ where $n_1,\dots, n_k \ne 0$. We call each $A_\lambda^{n_{2i+1}}$ and $B_\lambda^{n_{2i}}$ a \emph{syllable} of $w$ (see Definition \ref{definitions words}). Denote the upper right entry of the matrix $w$ by $p_w(\lambda)$. Thinking of $\lambda$ as a variable,
$p_w(\lambda)$ turns out to be a polynomial, whose coefficients depend on the integers $n_1,\dots,n_k$ and whose degree is precisely the number of syllables, $k$.
Moreover, finding some word $w$ as above such that $p_w(\lambda)=0$ is equivalent to showing that $\lambda$ is non-free (in the sense defined earlier).
This allows to find many non-free rationals by analyzing a corresponding infinite family of polynomials of low degree. Using only $k=3$, Lyndon and Ullman proved that for every positive integer $n$, the number $\frac{n}{n^2+1}$ is non-free \cite{lyndon}, and Beardon managed to improve this by showing that $\frac{n}{jn^2\pm1}$ is non-free for all positive integers $n$ and $j$ (with $nj\ge 2)$ \cite{beardon}. Variants of this method were subsequently developed by numerous authors, including \cite{tantan, bamberg, kim-koberda, new-seq} among others, to obtain more extensive families of non-free rationals.
While this has become the most prominent method, several other approaches have also emerged, sometimes being employed in conjunction with the former.
Linear recurrence sequences were used in \cite{lyndon, brenner-olesky}. More recently, non-free rationals $\frac{p}{q}$ were found in \cite{detinko2022freeness} by showing that $G_{\frac{p}{q}}$ has finite index in $G_{\frac{1}{q}}$,
and the structure of the group $G_{\frac{p}{q}, 1}$ for certain rationals was studied in the preprint \cite{nybergbrodda2024}, yielding that the index in $G_{\frac{1}{q}}$ is finite as well. In the preprint \cite{choi2024}, non-free numbers were characterized as roots of some family of polynomials.

All of these techniques seem to share a common drawback: they are effective only when relators with few syllables are present in $G_{\frac{p}{q}}$; this should not be surprising, considering that as mentioned earlier, such relators correspond to low degree polynomials.
For this reason, Beardon \cite[p.\ 532]{beardon} and the authors of \cite[p.\ 53]{detinko2022freeness}
have doubted that these techniques could be used to settle the conjecture for all $\Q\cap(-2,2)$, 
and it was indeed shown in \cite{gutan2014, choi2024} that for real numbers $\lambda$, relators in $G_\lambda$ consist of more syllables as $\lambda \to 2$.
In this paper we strengthen this result and provide further insights on the structure of relators in these groups.
Before presenting our results, let us introduce the following definitions and notation.
\begin{definition}\label{definitions words}
    Let $S$ be a finite set and $w$ a word composed of letters from $S$. Let $w = s_{1}^{n_1}\cdots s_{k}^{n_k}$ be its reduced form, i.e., $s_i\in S$, $s_i \ne s_{i+1}^{\pm 1}$ and $n_i \in \Z \setminus \{0\}$ for all $i$.
    The \emph{syllables} of $w$ are the subwords $s_1^{n_1},\dots, s_{k}^{n_k}$.
    We denote by $\syl_S(w)$ the number of syllables in $w$ (namely, $k$).
    Now let $G$ be a group generated by $S\subseteq G\setminus\{1\}$, then we define
    \begin{equation*}
        \sigma_S(G) = \min\{\syl_S(w) : w\text{ is a nonempty word composed of letters from }S \text{ and } w=1\text{ in } G\}.
    \end{equation*}
    In other words, $\sigma_S(G)$ is the minimal number of syllables in a relator; throughout this paper, every mention of a \emph{relator} refers to a \emph{nontrivial relator}, namely, a nonempty word in the free group. If $S$ freely generates $G$, we define $\sigma_S(G) = \infty$. When $S$ is obvious from the context, we omit it and simply write $\syl(w)$ and $\sigma(G)$.
\end{definition}
The following is our first main result. As mentioned, it shows that the number of syllables in relators of $G_{\lambda, 2}$ increases as $\lambda\to 2$ in $\C$, but moreover, that every relator must admit a long subword of a specific form. It also includes a similar result for the groups $G_{\lambda,i}$ with $\lambda\to 2$.
\begin{theorem}\label{complex theorem}
    Let $\mu \in \{2, i\}$. Then for every $N\in \N$ there exists a neighborhood $U\subseteq \C$ of $2$ such that for every $\lambda\in U$, every relator of $G_{\lambda,\mu}$ has a subword $w$ of the form:
    \begin{enumerate}
        \item if $\mu = 2$: $w = (A_\lambda B_2^{-1})^k$ or $w=(A_\lambda^{-1} B_2)^k$ where $k\ge N$,
        \item if $\mu = i$: $w=A_\lambda^{n_1}B_i^{n_2}\cdots A_\lambda^{n_{k-1}} B_i^{n_k}$ where $k\ge N$ and $n_j\in\{\pm1\}$.
    \end{enumerate}
    In particular, $\sigma(G_{\lambda,\mu})\to \infty$ as $\lambda\to 2$.
    Furthermore, a lower bound on the length of $w$ can be computed effectively. More precisely, let $\lambda_n\to 2$, then there are effectively computable constants $k_n\in \N$ such that $k_n \to \infty$, and any relator of $G_{\lambda_n, \mu}$ has a subword $w_n$ of the above form with $\syl(w_n)\ge k_n$.
\end{theorem}
This result does not depend on the specific chosen normalization of the groups $G_{\lambda,\mu}$, as the following corollary indicates. 
\begin{corollary}\label{corollary convergence}
    Let $(\lambda_n)_{n=1}^\infty,(\mu_n)_{n=1}^\infty \subseteq \C$ and suppose that $\lambda_n\cdot \mu_n\to \ell$ where $\ell\in\{\pm 4, \pm 2i\}$. Then for every $K\in \N$ there exists $N\in \N$ such that for every $n\ge N$, every relator of $G_{\lambda_n,\mu_n}$ has a subword $w$ of the form:
    \begin{enumerate}
        \item if $\ell = \pm 4$: $w = (A_{\lambda_n}\, B_{\mu_n}^{-1})^k$ or $w=(A_{\lambda_n}^{-1}\, B_{\mu_n})^k$ where $k\ge K$,
        \item if $\ell = \pm 2i$: $w=A_{\lambda_n}^{n_1}\,B_{\mu_n}^{n_2}\cdots A_{\lambda_n}^{n_{k-1}}\, B_{\mu_n}^{n_k}$ where $k\ge K$ and $n_j\in\{\pm1\}$.
    \end{enumerate}
    In particular, $\sigma(G_{\lambda_n,\mu_n})\underset{n\to\infty}{\longrightarrow}\infty$.
\end{corollary}
By Corollary \ref{corollary convergence}, if a complex sequence $(\mu_n)$ converges to $\mu\in\{\pm 4,\pm 2i\}$, then $\sigma(G_{1,\mu_n})\to\infty$.
The thing that makes those values of $\mu$ interesting, is that each of them is an accumulation point of $\{\mu'\in \C:G_{1, \mu'}  \text{ is non-free}\}$; it is not hard to see it for $\mu = \pm4$ (see Remark \ref{remark bound is sharp}), and for $\mu=\pm 2i$ it was proved by several authors \cite[Corollary 2]{chang-ree}, \cite[Corollary 2]{ree}, \cite[Theorem 5]{lyndon}.
To put it differently, there is no neighborhood $U$ of $\mu$ where $\sigma(G_{1,\mu'}) = \infty$ for all $\mu'\in U$.

In the case of real parameters, we are able to give an explicit lower bound on the length of the subword from Theorem \ref{complex theorem}.
\begin{theorem}\label{real theorem}
    For every real $1.8 \le \lambda < 2$, every (nontrivial) relator of $G_{\lambda,2}$ admits a subword of the form $(A_\lambda B_2^{-1})^k$ or $(A_\lambda^{-1}B_2)^k$, where
    \begin{equation*}
        k > \frac{1}{4\arccos(\sqrt{\lambda/2})} - 1.
    \end{equation*}
    In particular, $\sigma(G_{\lambda,2}) > 2(\frac{1}{4\arccos(\sqrt{\lambda/2})} - 1)$.
\end{theorem}
This asymptotic lower bound is sharp up to a multiplicative constant; see Remark \ref{remark bound is sharp} for the details.

As $\lambda\to 2$, the generators $A_\lambda, B_2$ of $G_{\lambda, 2}$ converge to $A_2, B_2$ which generate the free group of rank $2$. 
Theorem \ref{complex theorem} then shows that $\sigma(G_{\lambda, 2})\to \sigma(G_{2, 2}) = \infty$.
We stress that in general, $\sigma$ is not necessarily semicontinuous, as the following example demonstrates.
\begin{example}
    For $\lambda\in\R$ let $C_\lambda = \begin{pmatrix}
        1 & 2\\
        \frac{\lambda-1}{2} & \lambda
    \end{pmatrix}\in \SL_2(\R)$.
    Then $C_1 = A_2$, and since $A_2, B_2$ freely generate $G_2$, $\sigma(\langle C_1, B_2\rangle) = \sigma(G_2) = \infty$.
    However, let us show that $\liminf_{\lambda\to 1}\sigma(\langle C_\lambda, B_2\rangle) = 1$. We recall that every element in $\SL_2(\R)$ can be classified as either elliptic, parabolic or hyperbolic according to its trace. An element $g\in \SL_2(\R)$ with $\lvert \tr g\rvert = \lvert 2\cos\theta\rvert< 2$ is elliptic, and acts on the hyperbolic half-plane $\HH^2$ as a (hyperbolic) rotation of angle $\pm2\theta$; for more details, see for example \cite{beardon-book, katok, benedetti}. Thus, if  we define $\lambda_n = 2\cos(\pi/n) - 1$ for every $n\in\N$, then we obtain a sequence $\lambda_n\to 1$ such that $C_{\lambda_n}$ is a rotation of angle $2\pi/n$. Hence, $C_{\lambda_n}^n$ acts trivially on $\HH^2$, so that $C_{\lambda_n}^{2n} = 1$, and the assertion follows.
\end{example}

These results are based on two variations of the ping-pong lemma that we introduce (Props. \ref{prop subword} and \ref{prop subword length}), which can be applied to a group regardless of whether it is free or not. We will apply them to the actions of the group $G_{\lambda,\mu}$ on the real and complex projective lines.
The difficulty when dealing with syllables lies in the fact that a single syllable may consist of an arbitrarily large power, thereby making it harder to control the behaviour of (the non-equicontinuous family) $\{A_\lambda^n:n\ne 0\}$ as $\lambda$ varies.
However, one key observation is that the majority of group elements act nicely on those spaces.
The failure of the ping-pong lemma can only be attributed to very specific elements, which must be present in every relator.
This observation will permit us to confine our attention to only finitely many group elements and use continuity arguments.
In Section \ref{section ping pong} we will develop these variations of the ping-pong lemma and define relevant notions. This will be done in general settings, and in Section \ref{section results} we will revisit the groups $G_{\lambda,\mu}$ and prove the main results.
We conclude with a discussion of follow-up questions and potential directions for further investigation, as outlined in Section \ref{section discussion}.
\section{The ping-pong lemma}\label{section ping pong}
In this section, we aim to formulate and prove two variations of the ping-pong lemma: one for a single group and one for a family of groups.
We begin by introducing a particular decomposition of words in the free group, which will be employed in the proofs.
Let $F=F_S$ be the free group generated by some set $S$ with $2\le \lvert S\rvert <\infty$.
\begin{definition}
    Let $w = s_{1}^{n_1}\cdots s_{k}^{n_k}\in F$ be a reduced word with $s_i \in S$.
    An \emph{exact subword} of $w$ is any nonempty subword composed of syllables, i.e., of the form $s_i^{n_i}\cdots s_j^{n_j}$ where $i\le j$.
\end{definition}
\begin{definition}\label{definition forbidden+f decom}
    Let $W\subseteq F$ be any subset. Thinking of $W$ as forbidden words, we define $\mathcal{F}(W) \subseteq F$ to be the collection of all nonempty words avoiding $W$, namely
    \[\mathcal{F}(W) = \{w\in F: w\text{ is nonempty and no exact subword of }w \text{ belongs to }W\}.\]
    Note that $W\cap\mathcal{F}(W)=\emptyset$.
    An \emph{$\mathcal{F}(W)$-decomposition} of a word $w\in F$ is a decomposition $w = w_1\cdots w_n$ into exact subwords, such that for every $i$:
    \begin{enumerate}
        \item $w_i \in W\cup \mathcal{F}(W)$,
        \item no two consecutive subwords belong to $\mathcal{F}(W)$, i.e., if $w_i \in \mathcal{F}(W)$ then $w_{i+1}\in W$.
    \end{enumerate}
\end{definition}
\begin{lemma}\label{lemma W-F(W) decomposition}
    Let $W\subseteq F$ be any subset. Then every nonempty word $w\in F$ has an $\mathcal{F}(W)$-decomposition.
\end{lemma}
\begin{proof}
    Let $w\in F$ be nonempty. Consider the collection of all decompositions of $w$ into exact subwords $w = w_1\cdots w_n$  such that $\big\lvert\{1\le i\le n: w_i \in W\}\big\rvert=k$, where $k\in\N\cup\{0\}$ is the maximal possible\footnote{Note that the empty word is not considered an exact subword, and hence $w$ has only finitely many decompositions.}. Among this collection, choose a decomposition $w = w_1\cdots w_n$ such that $\big\lvert\{1\le i\le n: w_i \in \mathcal{F}(W)\}\big\rvert$ is the minimal possible; we claim that this is an $\mathcal{F}(W)$-decomposition. Indeed, if for some $i$, $w_i\notin \mathcal{F}(W)$, then there exists a decomposition $w_i=v_1\, u\, v_2$, where $u\in W$ is an exact subword of $w$, and each $v_j$ is either an exact subword or the empty word.
    By the maximality of $k$, the decomposition $w_1\cdots w_{i-1}\, v_1\, u\, v_2\, w_{i+1}\cdots w_n$
    contains no more subwords from $W$ than the former decomposition $w_1\cdots w_n$. This implies that $w_i\in W$, yielding property $(1)$.
    Next, suppose for contradiction that $w_i, w_{i+1}\in \mathcal{F}(W)$. Since this decomposition has the minimal possible number of subwords in $\mathcal{F}(W)$ (among the specified collection), we conclude that $w_i\, w_{i+1}\notin \mathcal{F}(W)$. This means that we can decompose $w_i\, w_{i+1} = v_1\, u\, v_2$
    where once again $u\in W$ is an exact subword and $v_1, v_2$ are exact subwords or empty.
    Replacing $w_i\, w_{i+1}$ with $v_1\, u\, v_2$, we obtain another decomposition of $w$ which contains more than $k$ subwords from $W$, contradicting the maximality of $k$.
\end{proof}
Let $H=\langle a,b\rangle$ be a group acting on a set $X$. Let $F_2$ be the free group generated by the symbols $\overline{a}$ and $\overline{b}$.
From now on it will be important to distinguish between the free group elements $\overline{w}=\overline{w}(\overline{a}, \overline{b})\in F_2$ and the corresponding element $w = \overline{w}(a,b) \in H$.

For $x,y\in \{a,b\}$ we denote by $F_2^{\overline{x},\overline{y}}$ the collection of all words in $F_2$ whose first letter is $\overline{x}$ and last letter is $\overline{y}$.
\begin{definition}\label{def W_K to X}
    For any four subsets $C_1, C_2, D_1, D_2\subseteq X$ define $W(C_1, C_2, D_1, D_2)\subseteq F_2$ to be the collection of all $\overline{w}\in F_2$ such that
    \begin{equation}\label{eq W_K to X}
        \begin{cases}
            w(C_1)\subseteq D_2 & \text{if } \overline{w}\in F_2^{\overline{a},\overline{a}}\\
            w(C_2)\subseteq D_2 & \text{if } \overline{w}\in F_2^{\overline{a},\overline{b}}\\
            w(C_1)\subseteq D_1 & \text{if } \overline{w}\in F_2^{\overline{b},\overline{a}}\\
            w(C_2)\subseteq D_1 & \text{if } \overline{w}\in F_2^{\overline{b},\overline{b}}.
        \end{cases}
    \end{equation}
\end{definition}
For $i=1,2$, let $K_i\subseteq X_i\subseteq X$ be some subsets. We will be interested in the two sets $W(K_1, K_2, X_1, X_2)$ and $W(X_1, X_2, K_1, K_2)$, and it will be helpful to think about them in the following context: for any $n\ne 0$, $a^n\in H$ ($b^n\in H$) is a ping-pong player whose goal is to pass elements of $X_1$ into $X_2$ (from $X_2$ into $X_1$). Every word can be thought of as an iteration of ping-pong shots, according to its syllables. For example, the goal of $b^{-1}a^2$ is to pass $X_1$ back into $X_1$. Since we will deal with non-free groups, some words will fail to achieve this.
The former set is the collection of words that still manage to pass “easier” elements, namely, elements of $K_1$ and $K_2$, even if they possibly fail to pass some other elements of $X_1$ and $X_2$. The latter set consists of the outstanding players who not only manage to make the passes from $X_1$ and $X_2$, but to put them inside $K_1$ and $K_2$.

The following is our first variant of the ping-pong lemma.
Essentially, it states that every relator of the group, whose first and last letters are identical, contains a subword $w$ that fails to pass elements from $K_1$ and $K_2$ into $X_1$ and $X_2$, and furthermore, no subword of $w$ succeeds in passing $X_1$ and $X_2$ into $K_1$ and $K_2$. This will prove useful in the next section, as we will see that with a suitable choice of $X_1,X_2,K_1,K_2$ there are relatively few words $w$ exhibiting these properties.
\begin{proposition}\label{prop subword}
    Let $H=\langle a,b\rangle$ be a group acting on a set $X$. Let $F_2$ be the free group generated by the symbols $\overline{a}$ and $\overline{b}$.
    Let $X_1, X_2 \subseteq X$ be disjoint, and $K_i\subseteq X_i$ be nonempty for $i=1,2$. Denote
    \[ W_{X\to K} = W(X_1, X_2, K_1, K_2) \subseteq F_2 \qquad \text{and} \qquad W_{K\to X} = W(K_1, K_2, X_1, X_2)\subseteq F_2.\]
    Then every relator of $H$, whose first and last letters are identical, has an exact subword that belongs to $\mathcal{F}(W_{X\to K})\setminus W_{K\to X}$.
    
    If in addition, there exists $N\in \N$ such that $\overline{a}^k\in W_{X\to K}$ for every $\lvert k\rvert \ge N$, then the assumption on the first and last letters of the relator can be dropped.
\end{proposition}
We will need the following lemma.
\begin{lemma}\label{lemma concat}
    Let $\overline{w}\in F_2$ be nonempty and $\overline{w}= \overline{w}_1\cdots \overline{w}_n$ be a decomposition into exact subwords that alternate between $W_{K\to X}$ and $W_{X\to K}$, i.e., either $\overline{w}_{2i}\in W_{K\to X}$ and $\overline{w}_{2i+1}\in W_{X\to K}$ or the other way around. Then $\overline{w}\in W_{K\to X}$.
\end{lemma}
\begin{proof}[Proof of Lemma \ref{lemma concat}]
    The proof proceeds by induction on $n$. Since $K_i\subseteq X_i$ for $i=1,2$, we have $W_{X\to K}\subseteq W_{K\to X}$, and thus the base case $n=1$ follows from the assumption.
    Now, suppose we are given a word $\overline{w}$ as above with $n>1$. Assume further that $\overline{w}_n\in W_{K\to X}$ (and hence $\overline{w}_{n-1}\in W_{X\to K}$), and that the first and last letters of $\overline{w}_{n}$, as well as the first letter of $\overline{w}_{n-1}$, are $\overline{a}$; the proof for all other cases is analogous. Since $\overline{w}_{n-1}$ and $\overline{w}_n$ are exact subwords, the last letter of $\overline{w}_{n-1}$ must differ from the first letter of $\overline{w}_n$, implying that it must be $\overline{b}$. This gives us the inclusion
    \[w_{n-1}w_n(K_1) \subseteq w_{n-1}(X_2) \subseteq K_2.\]
    If $\overline{w}=\overline{w}_{n-1}\overline{w}_n$ then the proof is complete. Otherwise, by the same reasoning, the last letter of $\overline{w}_{n-2}$ must be $\overline{b}$. Applying the induction hypothesis and the previous inclusion, we obtain
    \[w_1\cdots w_n(K_1)\subseteq w_1\cdots w_{n-2}(K_2)\subseteq X_i,\]
    where $i\in\{1,2\}$ depends on the first letter of $\overline{w}_1$.
\end{proof}
\begin{proof}[Proof of Prop. \ref{prop subword}]
    Let $\overline{w}\in F_2$ be a nonempty word that has no exact subword from $\mathcal{F}(W_{X\to K})\setminus W_{K\to X}$, and let $\overline{w}= \overline{w}_1\cdots \overline{w}_n$ be some $\mathcal{F}(W_{X\to K})$-decomposition (Lemma \ref{lemma W-F(W) decomposition}).
    We will show by induction on $\big\lvert\{i:\overline{w}_i \notin W_{X\to K}\}\big\rvert$ that $\overline{w}\in W_{K\to X}$. If $\overline{w}_i\in W_{X\to K}$ for all $i$, then $\overline{w}\in W_{K\to X}$; this can be seen as a simple case of Lemma \ref{lemma concat} since $W_{X\to K}\subseteq W_{K\to X}$.
    Next, suppose that for some $i$, $\overline{w}_i\notin W_{X\to K}$. By the first property of the $\mathcal{F}(W_{X\to K})$-decomposition (Definition \ref{definition forbidden+f decom}) $\overline{w}_i\in \mathcal{F}(W_{X\to K})$ and by our assumption on $\overline{w}$, $\overline{w}_i\in W_{K\to X}$.
    If $i=1$ let us define $\overline{w}_{i-1}$ to be the empty word, and otherwise,
    the second property of the $\mathcal{F}(W_{X\to K})$-decomposition implies that $\overline{w}_{i-1}\in W_{X\to K}$. Similarly, $\overline{w}_{i+1}$ is either empty if $i=n$ or belongs to $W_{X\to K}$ as well.
    If $i>2$, then $\overline{u}_1\coloneqq\overline{w}_1\cdots\overline{w}_{i-2}\in W_{K\to X}$ by the induction hypothesis and if $i\le 2$ we let $\overline{u}_1$ be the empty word. Similarly, $\overline{u}_2\coloneqq\overline{w}_{i+2}\cdots \overline{w}_n\in W_{K\to X}$ if $i<n-1$ and $\overline{u}_2$ is empty otherwise.
    Decomposing $\overline{w}$ as $\overline{w} = \overline{u}_1\,\overline{w}_{i-1}\,\overline{w}_i\,\overline{w}_{i+1}\,\overline{u}_2$,
    we deduce from Lemma \ref{lemma concat} that $\overline{w} \in W_{K\to X}$, completing the induction.
    In particular, if $\overline{w}\in F_2^{\overline{a},\overline{a}}$ then $w(K_1)\subseteq X_2$ which implies that $w \ne 1$, and a similar argument applies if $\overline{w}\in F_2^{\overline{b},\overline{b}}$.
    
    For the second part, let $\overline{w}$ be a relator of $H$. Choose $m$ such that the first and last syllables of $\overline{w}' \coloneqq \overline{a}^m\, \overline{w}\, \overline{a}^{-m}$ are $\overline{a}^k, \overline{a}^{k'}$ for some $\lvert k \rvert, \lvert k' \rvert \ge N$. Then $\overline{w}'$ has an exact subword from $\mathcal{F}(W_{X\to K})\setminus W_{K\to X}$, and it must also be an exact subword of $\overline{w}$.
\end{proof}
\begin{corollary} \label{corollary good words}
    Keeping the notation of Proposition \ref{prop subword}, suppose that
    \begin{equation*}
        \{\overline{a}^n, \overline{b}^n: \lvert n\rvert\ge 2\}\subseteq W_{X\to K},
    \end{equation*}
    and denote
    \begin{equation}\label{set A}
        \mathcal{A} = \{\overline{w}\in F_2:\text{all syllables of }\overline{w}\text{ are either }\overline{a},\overline{a}^{-1},\overline{b} \text{ or }\overline{b}^{-1}\}
    \end{equation}
    (e.g.\ $w=aba^{-1}b$).
    Then every relator of $H$ has an exact subword that belongs to $\mathcal{A}\setminus W_{K\to X}$.
    
    If in addition
    \begin{equation}\label{eq additional assumption}
        \{(\overline{a}\overline{b})^{\pm1}, (\overline{b}\overline{a})^{\pm1}\}\subseteq W_{X\to K},
    \end{equation}
    Then every relator of $H$ has an exact subword that belongs to $\mathcal{B}\cup\mathcal{B}^{-1}\setminus W_{K\to X}$, where
    \begin{equation}\label{set B}
        \mathcal{B} = \{(\overline{a}\overline{b}^{-1})^k,\quad (\overline{b}^{-1}\overline{a})^k,\quad (\overline{a}\overline{b}^{-1})^{k-1}\overline{a},\quad (\overline{b}^{-1}\overline{a})^{k-1}\overline{b}^{-1}: k>0\}
    \end{equation}
    and $\mathcal{B}^{-1} = \{\overline{w}^{-1}:\overline{w}\in\mathcal{B}\}$.
\end{corollary}
\begin{proof}
    In every word that admits no exact subword from $W_{X\to K}$, every letter $\overline{a}^{\pm1}$ can be followed only by $\overline{b}^{\pm1}$ and vice versa. Thus, $\mathcal{F}(W_{X\to K}) \subseteq \mathcal{A}$.
    Under the additional assumption \eqref{eq additional assumption}, in every word in $\mathcal{F}(W_{X\to K})$ each letter $\overline{a}, \overline{a}^{-1}, \overline{b}, \overline{b}^{-1}$ can be followed only by the letter $\overline{b}^{-1},\overline{b}, \overline{a}^{-1}, \overline{a}$ respectively. Therefore in this case, $\mathcal{F}(W_{X\to K}) \subseteq \mathcal{B}\cup\mathcal{B}^{-1}$. The corollary now follows from Proposition \ref{prop subword}.
\end{proof}
Until now, we have focused solely on a single group generated by the elements $a,b$. We will now consider the scenario where $a,b$ vary.
To be more precise, let $H$ be a topological group and $\Lambda$ a metric space. Then we let $a,b:\Lambda\to H$ be continuous functions. Given $\lambda\in\Lambda$, it will be convenient to denote $a_\lambda\coloneqq a(\lambda)$ and $b_\lambda\coloneqq b(\lambda)$.
Suppose that $H$ acts continuously on a metric space $X$. For subsets $X_1, X_2, K_1, K_2\subseteq X$, let us denote
\[ W_{X\to K}^\lambda = W(X_1, X_2, K_1, K_2) \qquad \text{and} \qquad W_{K\to X}^\lambda = W(K_1, K_2, X_1, X_2).\]
where the right-hand side of both is defined with respect to the group $H_\lambda\coloneqq \langle a_\lambda, b_\lambda\rangle$ (see Definition \ref{def W_K to X}). We recall that these are subsets of the free group $F_2$ generated by the symbols $\overline{a}$ and $\overline{b}$. Here is our second variant of the ping-pong lemma.
\begin{proposition}\label{prop subword length}
    In the above settings, suppose that $X_1$ and $X_2$ are disjoint and open, and that $K_i\subseteq X_i$ are nonempty compact subsets for $i=1,2$.
    Suppose further that for some $\lambda_0\in \Lambda$ and every $n\ne 0$, $a_{\lambda_0}^n X_1 \subseteq X_2$ and $b_{\lambda_0}^n X_2 \subseteq X_1$.
    Then for every $N\in \N$ there exists a neighborhood $U\subseteq \Lambda$ of $\lambda_0$ with the following property:
    let $\lambda\in U$, and $\overline{w}$ a relator of $H_\lambda$ whose first and last letters are identical.
    Then $\overline{w}$ has an exact subword from $\mathcal{F}(W_{X\to K}^\lambda)$ whose length is at least $N$.
    If in addition, there exists $M\in \N$ such that $\overline{a}^k\in W_{X\to K}^\lambda$ for every $\lvert k\rvert \ge M$ and $\lambda\in\Lambda$, then the assumption on the first and last letters of $\overline{w}$ can be dropped.
\end{proposition}
\begin{proof}
    First we recall that $\len(\overline{w})$ denotes the number of letters of the word $\overline{w}$, and that $F_2^{\overline{x},\overline{y}}\subseteq F_2$ consists of the words that begin with the letter $\overline{x}$ and end with $\overline{y}$. 
    Let $N\in \N$, and define $L = \{w\in F_2: \len(w)< N\}$.
    For $x,y\in\{a,b\}$, we denote $L^{\overline{x},\overline{y}} = F_2^{\overline{x},\overline{y}}\cap L$, and for $\lambda\in \Lambda$, $L_\lambda^{x,y}$ is the subset of corresponding elements in $H_\lambda$ (i.e., the image of $L^{\overline{x},\overline{y}}$ under the homomorphism $\overline{a}\mapsto a_\lambda, \overline{b}\mapsto b_\lambda$). Denote by $d$ the metric on $X$.
    Since $L$ is finite, $K_1$ is compact and $a:\Lambda\to H$ is continuous, the function
    \begin{equation*}
        h:\Lambda\to \R,\qquad h(\lambda) = d(L^{a, a}_{\lambda}K_1, X\setminus X_2)
    \end{equation*}
    (where $L^{a, a}_{\lambda}K_1 \coloneqq \{gx:g\in L^{a, a}_{\lambda}, x\in K_1\}$) is continuous.
    By our assumption on $\lambda_0$, $L^{a, a}_{\lambda_0}K_1 \subseteq X_2$ and hence $h>0$ on some neighborhood $U\subseteq\Lambda$ of $\lambda_0$. In other words, $L^{a, a}_{\lambda}K_1\subseteq X_2$ for every $\lambda\in U$.
    Arguing analogously and taking smaller $U$ if necessary, we can see that for every $\lambda\in U$,
    \begin{equation}\label{eq containment k-x2}
        L^{a, a}_{\lambda}K_1\subseteq X_2,\qquad L^{a, b}_{\lambda}K_2\subseteq X_2,\qquad L^{b, a}_{\lambda}K_1 \subseteq X_1 \qquad \text{and} \qquad L^{b, b}_{\lambda}K_2\subseteq X_1.
    \end{equation}
    This means that for every $\lambda\in U$, $L \subseteq W_{K\to X}^\lambda$. In other words, all words in $\mathcal{F}(W_{X\to K}^\lambda)\setminus W_{K\to X}^\lambda$ are of length at least $N$. 
    We conclude by applying Proposition \ref{prop subword}.
\end{proof}
\begin{corollary}
    Under the assumptions of the Proposition \ref{prop subword length}, suppose that there exists $M\in \N$ such that $\overline{a}^k,\overline{b}^k\in W_{X\to K}^\lambda$ for every $\lvert k\rvert \ge M$ and $\lambda\in\Lambda$. Then $\sigma(H_\lambda)\to\infty$ as $\lambda\to\lambda_0$ (cf.\ Definition \ref{definitions words}).
\end{corollary}
\begin{proof}
    Let $\lambda_n\to \lambda_0$ and for every $n$ let $\overline{w}_n$ be a relator of $H_{\lambda_n}$. Then by Proposition \ref{prop subword length}, there is a sequence $(\overline{u}_n)$ such that every $\overline{u}_n$ is an exact subword of $\overline{w}_n$ that belongs to $\mathcal{F}(W_{X\to K}^{\lambda_n})$ and $\len(\overline{u}_n)\to \infty$.
    By our assumption, every syllable of $\overline{u}_n$ is of the form $a^k$ or $b^k$ where $\lvert k\rvert < M$.
    Since for any $C\in\N$, there are only finitely many such words with less than $C$ syllables, it follows that
    $\syl(\overline{w}_n)\ge \syl(\overline{u}_n)\underset{n\to\infty}{\longrightarrow}\infty$. Since $(\lambda_n)$ and $(\overline{w}_n)$ were arbitrary, $\sigma(H_\lambda)\to\infty$.
\end{proof}

\begin{corollary}\label{cor long ab}
    Under the assumptions of Proposition \ref{prop subword length},
    suppose that for every $\lambda\in\Lambda$
    \begin{equation}\label{eq forbidden words 1}
        \{\overline{a}^n, \overline{b}^n: \lvert n\rvert\ge 2\}\subseteq W_{X\to K}^\lambda.
    \end{equation}
    Then for every $N\in \N$ there exists a neighborhood $U\subseteq \Lambda$ of $\lambda_0$ such that for every $\lambda\in U$, every relator of $H_\lambda$ has an exact subword of the form $\overline{a}^{n_1}\overline{b}^{n_2}\cdots \overline{a}^{n_{k-1}} \overline{b}^{n_k}$ where $k\ge N$ and $n_j\in\{\pm1\}$.

    If in addition,
    \begin{equation}\label{eq forbidden words 2}
        \{(\overline{a}\overline{b})^{\pm1}, (\overline{b}\overline{a})^{\pm1}\}\subseteq W_{X\to K}^\lambda,
    \end{equation}
    then the form of the mentioned subword is $(\overline{a}\overline{b}^{-1})^{k'}$ or $(\overline{a}^{-1}\overline{b})^{k'}$ where $2k'\ge N$.
\end{corollary}
\begin{proof}
    Let $U\subseteq \Lambda$ be the neighborhood of $\lambda_0$ given by Proposition \ref{prop subword length}, corresponding to $N+2$ rather than $N$. Let $\lambda\in U$ and let $w$ be a relator of $H_{\lambda}$.
    As was seen in the proof of Corollary \ref{corollary good words}, $\mathcal{F}(W_{X\to K}^\lambda) \subseteq \mathcal{A}$ (where $\mathcal{A}$ is as in Equation \eqref{set A}), and hence by Proposition \ref{prop subword length}, $w$ has an exact subword from $\mathcal{A}$ whose length is at least $N+2$. If this subword begins with the letter $b^{\pm 1}$, or ends with $a^{\pm 1}$, we can omit these letters to obtain a subword of the desired form.
    The second part follows similarly, using the fact that $\mathcal{F}(W_{X\to K}^\lambda) \subseteq \mathcal{B}\cup\mathcal{B}^{-1}$ (where $\mathcal{B}$ is defined in Equation \eqref{set B}).
\end{proof}

\section{Proofs of the main results}\label{section results}
In this section, we shift our focus back to the groups defined in the introduction, and prove all of the main results.
Recall the notation $A_\lambda, B_\lambda$ from Equation \eqref{matrix-notation}.
We consider the standard action of $\SL_2(\C)$ by M{\"o}bius transformations on the extended complex plane $\widehat{\C}\coloneqq\C\cup\{\infty\}$.
In particular, $A_\lambda^n z = z + n\lambda$ and $B_\lambda^n z = \frac{z}{n\lambda z+1}$ for every $z\in \widehat{\C}$.
We begin with a simple observation.
\begin{lemma}\label{symmetry lemma}
    Let $\lambda,\mu\in \C$ and let $w = s_1^{n_1}\cdots s_k^{n_k}$ be a reduced word with $s_i\in \{A_\lambda, B_\mu\}$. Then using the convention $-\infty=\infty$, for every $z\in\widehat{\C}$,
    \[s_1^{-n_1}\cdots s_k^{-n_k}(z) = -w(-z).\]
\end{lemma}
\begin{proof}
    The claim is readily checked for $w= A_\lambda^n$ or $w = B_\mu^n$, and hence follows by induction on $\syl(w)$.
\end{proof}

\begin{proof}[Proof of Theorem \ref{complex theorem} for $\mu = 2$]
Let $\Lambda = \{\lambda\in \C: \re \lambda > 1.5\}$, and define $a,b:\Lambda\to \SL_2(\C)$ by
\begin{equation}\label{eq a b functions}
    a_\lambda \coloneqq a(\lambda) = \begin{pmatrix}
    1 &  \lambda \\
    0 & 1
    \end{pmatrix} = A_\lambda,\qquad  b= b(\lambda) = \begin{pmatrix}
    1 & 0 \\
    2 & 1
    \end{pmatrix} = B_2
\end{equation}
(the function $b$ is constant). Define
\begin{equation}\label{eq X1 X2}
    X_1 = \{z\in \C:\lvert z\rvert < 1\}\subseteq \widehat{\C},\qquad X_2 = \{z\in \C:\lvert z\rvert > 1\}\cup\{\infty\} \subseteq \widehat{\C}.
\end{equation}
We note that for every $n\ge 1$ and $\lambda\in \Lambda$,
\begin{equation}\label{eq a^n}
    \re (a_\lambda^n z) = \re z+n\re\lambda
\end{equation}
and also
\begin{align}
\begin{minipage}{.7\textwidth}
\centering
$b^n(X_2)$ is the open ball centered on the real axis and whose boundary contains $\frac{1}{2n+1}$ and $\frac{1}{2n-1}$.\end{minipage} \label{eq b^n}
\end{align}
It follows from these observations that for every $n\ge 1$, $a_2^n (X_1)\subseteq X_2$ and $b^n (X_2)\subseteq X_1$.
By Lemma \ref{symmetry lemma} and the symmetry (about $0$) of $X_1$ and $X_2$, the same holds for $n<0$ as well, verifying the assumption of Proposition \ref{prop subword length} for $\lambda_0 = 2$.

Next, we define
\begin{equation}\label{eq K1 K2}
    K_1 = \{\lvert z\rvert \le 1/2\},\qquad K_2 = \{\lvert z\rvert \ge 5/4\}\cup\{\infty\},
\end{equation}
and turn to verify Equations \eqref{eq forbidden words 1} and \eqref{eq forbidden words 2}.
Let $\lambda\in \Lambda$. Once again by Lemma \ref{symmetry lemma}, it suffices to verify those inclusions for the words with positive powers. Let $n\ge 2$. By Equation \eqref{eq a^n}, $a_\lambda^n (X_1)\subseteq K_2$ and by Equation \eqref{eq b^n}, the leftmost point of $b^n (X_2)$ is positive and its rightmost point is bounded above by $1/3$, so $b^n (X_2) \subseteq K_1$.
Equation \eqref{eq b^n} also implies that for $z\in X_2$, $\re bz > 1/3$, and thus $a_\lambda b(X_2) \subseteq K_2$.
Finally,
\[ba_\lambda(X_1)\subseteq b(\{\re z > 0\}) = \{\lvert z - 1/4 \rvert <1/4\}\subseteq K_1.\]
Theorem \ref{complex theorem} with $\mu = 2$ is now an immediate consequence of Corollary \ref{cor long ab} with $\lambda_0 = 2$.

Let us explain why given $\lambda_n\to 2$, we can compute the constants $k_n$ mentioned in the theorem effectively.
To find a lower bound on the length of the exact subword $w_n$ in the theorem, one needs to follow the proofs, which involve determining whether $w(K_1)$ or $w(K_2)$ intersect $X_1$ and $X_2$, for increasingly longer words $w\in \mathcal{B}$ (see Equation \eqref{set B}), until $w\notin W_{K\to X}$. Since the boundaries of these subsets in $\widehat{\C}$ are either lines or circles, this requires computing the image of a finite collection of points under $w$. We also note that our procedure is guaranteed to terminate only when $\lambda_n$ corresponds to a non-free group (since the existence of $w\in \mathcal{B}\setminus W_{K\to X}$ arises from the presence of a relator). Therefore, to ensure that $k_n$ is computed after finitely many steps, the procedure should be stopped after a prescribed number of iterations, which depends on $\lambda_n$ and increases as $n\to \infty$ (so that $k_n\to\infty$).
\end{proof}
We postpone the proof of the second part of the theorem (with $\mu=i$), and turn to prove Theorem \ref{real theorem}.
Fix some $\lambda\in\R$ and define $a_\lambda$ and $b$ as in Equation \eqref{eq a b functions}.
Then $a_\lambda, b \in \SL_2(\R)$, which acts (by M{\"o}bius transformations) on $\R\cup\{\infty\}$ (the real projective line).
We take the intersection of $\R\cup\{\infty\}$ with $X_1,X_2,K_1$ and $K_2$ defined in Equations \eqref{eq X1 X2} and \eqref{eq K1 K2} and keep the same notation for these subsets, namely
\begin{align*}
    &X_1 = \{x\in \R:-1< x < 1\}, &X_2& = \{x\in \R:\lvert x\rvert > 1\}\cup\{\infty\},\\
    &K_1 = \{x\in \R:-\frac{1}{2}\le x\le \frac{1}{2}\}, &K_2& = \{x\in \R:\lvert x\rvert \ge \frac{5}{4}\}\cup\{\infty\}.
\end{align*}
We wish to understand the set $W_{K\to X}^\lambda$, in order to use Corollary \ref{corollary good words}.
\begin{lemma} \label{lemma K2 to X2}
    Let $1.8\le \lambda < 2$ and choose $0<\theta < \pi/4$ such that $\lambda = 2\cos^2\theta$. Then
    \begin{equation}\label{eq ab^-1 K2 to X2}
        (a_\lambda b^{-1})^k (K_2) \subseteq X_2 \text{ for every integer }0\le k\le \frac{1}{4\theta}.
    \end{equation}
\end{lemma}
\begin{proof}
    It will be beneficial to use geometric considerations arising from the action of $\SL_2(\R)$ on the hyperbolic half-plane $\HH^2 = \{z\in \C:\im z > 0\}$. Denote $c_\lambda = a_\lambda b^{-1}$. As $\lvert \tr c_\lambda\rvert < 2$, $c_\lambda$ is elliptic and hence acts as a rotation about some point in $\HH^2$.
    We recall that a hyperbolic rotation of angle $2\alpha$ about $z\in \HH^2$ is represented by
    \begin{equation}\label{eq rotation formula}
        \pm\begin{pmatrix}
        \overline{z} &  z \\
         1 & 1 \\
        \end{pmatrix}
      \begin{pmatrix}
        e^{i\alpha} &   \\
          & e^{-i\alpha} \\
      \end{pmatrix}
        \begin{pmatrix}
        \overline{z} &  z \\
         1 & 1 \\
      \end{pmatrix}^{-1},
    \end{equation}
    and a simple calculation reveals that $c_\lambda$ is a rotation of angle $-4\theta$ about $z_\theta\coloneqq\frac{1 + e^{2i\theta}}{2}$.
    We claim that \eqref{eq ab^-1 K2 to X2} holds if and only if
    \begin{equation}\label{eq 5/4}
        c_\lambda ^{k} (5/4) \in X_2 \text{ for every integer }0\le k\le \frac{1}{4\theta}. 
    \end{equation}
    Since $5/4\in K_2$, \eqref{eq ab^-1 K2 to X2} immediately implies \eqref{eq 5/4}. The other direction
    should be intuitively clear as well: since $c_\lambda$ is a clockwise rotation, the first point of $K_2$ to hit $X_2^c = \{-1\le x\le 1\}$ when iterating $c_\lambda$ should be the leftmost positive point in $K_2$, which is $5/4$; we just need to verify that $5/4$ does not “jump over” $X_2^c$ when rotated.
    For $x,y\in \overline{\HH^2}$, we denote by $[x,y]$ the hyperbolic geodesic segment between $x$ and $y$.
    Let $\beta$ be the angle between $[z_\theta, 1]$ and $[z_\theta,-1]$ in $\HH^2$, measured clockwise; see Figure \ref{fig-beta}.
    \begin{figure}
    \centering
    \includegraphics[width=1\textwidth]{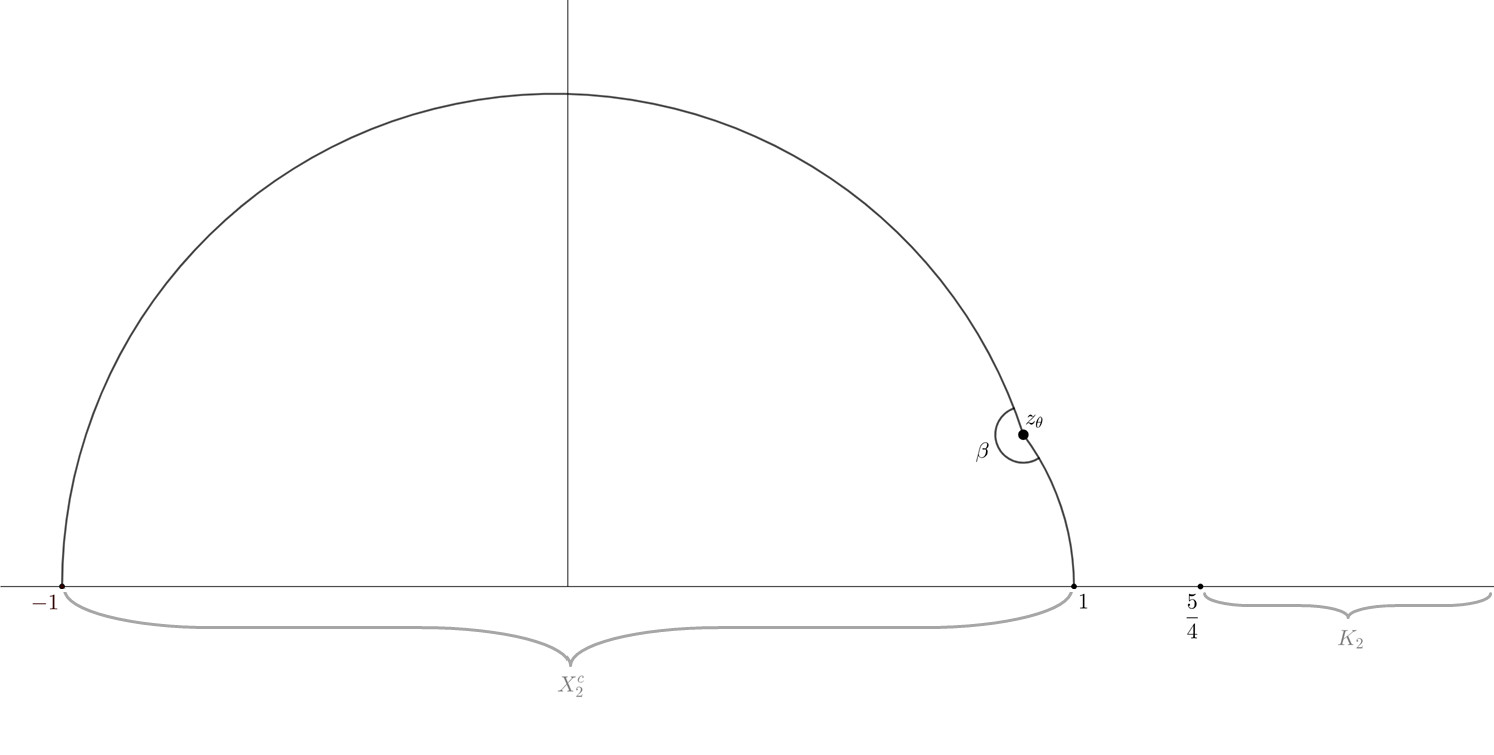}
    \caption{\label{fig-beta}The angle between the hyperbolic geodesics $[z_\theta, 1]$ and $[z_\theta,-1]$, denoted $\beta$, is greater than $\pi$.}
    \end{figure}
    Observe that $z_\theta$ lies on the geodesic $[0,1]$. Consequently, $\pi <\beta < 2\pi$, and in particular, $4\theta < \beta$ which shows the equivalence of \eqref{eq ab^-1 K2 to X2} and \eqref{eq 5/4}.
    
    Let $\alpha$ be the angle between the geodesics $[z_\theta,5/4]$ and $[z_\theta,1]$, measured clockwise (Figure \ref{fig-alpha}).
    \begin{figure}
    \centering
    \includegraphics[width=1\textwidth]{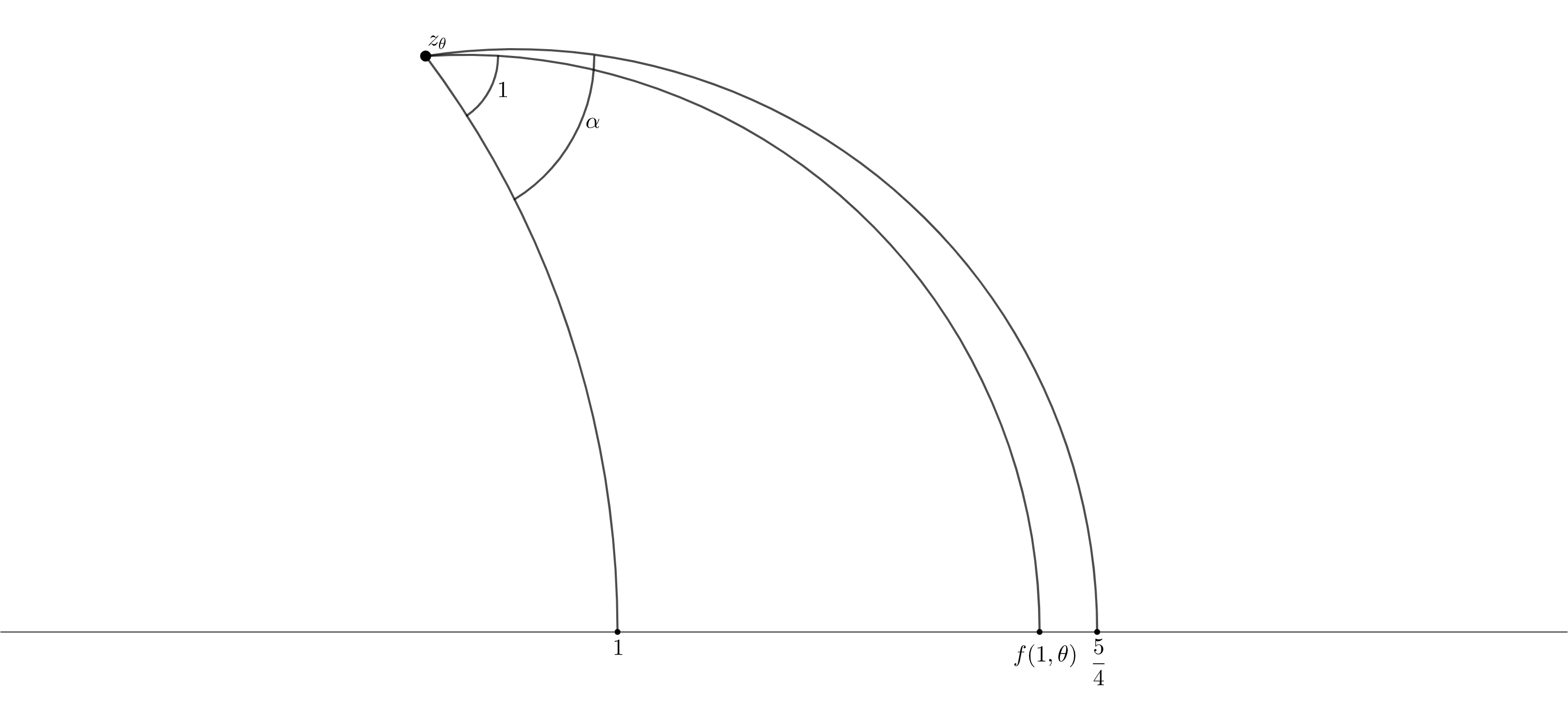}
    \caption{\label{fig-alpha}The angle between $[z_\theta,\frac{5}{4}]$ and $[z_\theta,1]$, denoted $\alpha$, is greater than $1$.}
    \end{figure}
    Then $\alpha > 1$; one way to see it is as follows: for $x\in\R$, let $f(x,\theta)$ denote the (hyperbolic) rotation of $x$ by an angle of $1$ about the point $z_\theta$. Using Equation \eqref{eq rotation formula} for $f$, we see that
    \[ f(1,\theta) = \frac{\cos(0.5)\cos(\theta)}{\cos(0.5+\theta)}.\]
    The maximum of $f(1,\theta')$ over all $0<\theta'<\pi/4$ such that $\lambda' = 2\cos^2 \theta' \ge 1.8$ is attained on $\theta_0$ which corresponds to $\lambda_0 = 1.8$, and $f(1,\theta_0) \approx 1.22$. Thus, $f(1,\theta)< 5/4$, implying that $\alpha > 1$. Therefore, for every integer $0\le k\le \frac{1}{4\theta}$,
    \[4\theta\cdot k \le 1 < \alpha,\]
    which means that $c_\lambda ^{k}(5/4) \in X_2$ as required.
\end{proof}
\begin{proof}[Proof of Theorem \ref{real theorem}]
    Let $1.8 \le \lambda < 2$, and choose $0 <\theta <\pi/4$ such that $\lambda = 2\cos^2(\theta)$. Let $0\le k \le \frac{1}{4\theta}$ be an integer, and denote $u_\lambda = (a_\lambda b^{-1})^k$. By Lemma \ref{lemma K2 to X2}, $u_\lambda (K_2)\subseteq X_2$.
    Observe that $a_\lambda (K_1)\subseteq K_2$, and recall from the proof of Theorem \ref{real theorem} that $b^{-1}(X_2) \subseteq X_1$. Hence, we obtain
    \begin{align*}
        b^{-1}u_\lambda (K_2)&\subseteq b^{-1}(X_2) \subseteq X_1, \\
        u_\lambda a_\lambda (K_1)&\subseteq  u_\lambda (K_2) \subseteq X_2,\\
        (b^{-1}a_\lambda)^{k+1} (K_1)&= b^{-1}u_\lambda a_\lambda (K_1) \subseteq X_1.
    \end{align*}
    In other words,
    \begin{equation*}
        \{ w\in \mathcal{B}: \syl(w)\le \frac{1}{2\theta}\}\subseteq W_{K\to X}^\lambda,
    \end{equation*}
    where $\mathcal{B}$ is the subset from Equation \eqref{set B}.
    By Lemma \ref{symmetry lemma}, every $w\in\mathcal{B}^{-1}$ with $\syl(w)\le\frac{1}{2\theta}$ belongs to $W_{K\to X}^\lambda$ as well.
    It follows from Corollary \ref{corollary good words} that every relator of $H_\lambda = G_{\lambda, 2}$ has an exact subword from $\mathcal{B}\cup \mathcal{B}^{-1}$ with more than $\frac{1}{2\theta}$ syllables.
    By possibly omitting the first and last letters, we obtain an exact subword of the desired form and length, completing the proof.
\end{proof}
\begin{remark}\label{remark bound is sharp}
    The lower bound in Theorem \ref{real theorem} cannot be improved beyond a constant factor, without further restrictions on $\lambda$:
     for every integer $n>2$, choose $\theta_n= \pi/2n$. As before, let $\lambda_n = 2\cos^2(\theta_n)$ and $c_{\lambda_n} = a_{\lambda_n}b^{-1}$. Then $c_{\lambda_n}$ is a (hyperbolic) rotation of angle $-4\theta_n=-2\pi/n$, which implies that $c_{\lambda_n}^{n}$ acts trivially on $\HH^2$ and that $c_{\lambda_n}^{2n}$ is a relator with $4n$ syllables, and hence
     \[ \sigma(G_{\lambda_n, 2})\le 4n = \frac{2\pi}{\theta_n}=\frac{2\pi}{\arccos(\sqrt{\lambda_n/2})}.\]
 \end{remark}

\begin{proof}[Proof of Theorem \ref{complex theorem} for $\mu = i$]
    Let $\Lambda = \{\lambda\in \C: \re \lambda > 1.5\}$ as before, and define $a,b:\Lambda\to \SL_2(\C)$ by
\begin{equation}
    a_\lambda = \begin{pmatrix}
    1 &  \lambda \\
    0 & 1
    \end{pmatrix},\qquad  b=b_\lambda = \begin{pmatrix}
    1 & 0 \\
    i & 1
    \end{pmatrix},\qquad \lambda\in\Lambda.
\end{equation}
For $z\in\C$ and $r>0$ denote $\mathfrak{B}_r(z) = \{w\in \C:\lvert z-w\rvert<r\}$, and define
    \[X_1 = \mathfrak{B}_1(i)\cup \mathfrak{B}_1(0)\cup \mathfrak{B}_1(-i), \qquad X_2 = \{z\in \C: \lvert\re z\rvert > 1\}\cup \{\lvert z\rvert>2\}\cup \{\infty\}.\]
    It is easy to see that $a_2^n (X_1)\subseteq X_2$ for every $n\ne 0$, and that we can choose some compact $K_2\subseteq X_2$ such that $a_\lambda^n (X_1) \subseteq K_2$ for every $\lvert n\rvert \ge 2$ and $\lambda\in \Lambda$.
    Following \cite[Theorem 2]{lyndon}, one simple geometric way to see the inclusions for $b$ is to set $J = \begin{pmatrix}
    0 &  -1 \\
    1 & 0
    \end{pmatrix}$, $b'z= (Jb^{-1}J^{-1})z = z + i$ and
    \[X_1' = JX_1 = \{\lvert \im \rvert > 1/2\}\cup \{\lvert z\rvert>1\} \cup\{\infty\},\qquad X_2' = JX_2 = \mathfrak{B}_\frac{1}{2}(\frac{1}{2})\cup \mathfrak{B}_\frac{1}{2}(0)\cup \mathfrak{B}_\frac{1}{2}(-\frac{1}{2}).\]
    Then once again it is clear that $(b')^n (X_2')\subseteq X_1'$ for every $n\ne 0$, and that we can choose some compact $K_1'\subseteq X_1'$ such that $(b')^n (X_2') \subseteq K_1'$ for every $\lvert n\rvert \ge 2$. The analogous inclusions for $b^n$ now follow with $K_1 = J^{-1}K_1'$, and we conclude by applying Corollary \ref{cor long ab}.
    The constants $(k_n)$ mentioned in the theorem can be computed effectively for reasons similar to those discussed in the proof of the first part of the theorem.
\end{proof}
\begin{proof}[Proof of Corollary \ref{corollary convergence}]
Since $G_{\lambda_n,\mu_n}=G_{\lambda_n,-\mu_n}$, by replacing $\mu_n$ with $-\mu_n$ if necessary we may assume that $\ell \in \{4,2i\}$.
For every (sufficiently large) $n$, choose $z_n\in\C$ such that $\mu_n z_n^{-2} = \ell/2$. Then
    \[J_nB_{\mu_n}J_n^{-1} = B_{\ell/2},\quad J_nA_{\lambda_n}J_n^{-1} = A_{\lambda_n'},\qquad \text{where } J_n= \begin{pmatrix}
        z_n & 0\\
        0 & z_n^{-1}
    \end{pmatrix} \text{ and } \lambda_n' = \lambda_n z_n^2.\]
    Then $\lambda_n' \to 2$ and  the assertion follows from Theorem \ref{complex theorem} and the fact that $J_n$ conjugates the relators of $G_{\lambda_n,\mu_n}$ to those of $G_{\lambda_n',\ell/2}$.
\end{proof}
\section{Discussion and questions}\label{section discussion}
We conclude the paper by presenting a few questions that came up during this work, along with potential future research directions.
\begin{enumerate}
\item Let us briefly turn to the normalization $G_{1,\mu},\; \mu\in \C$, which is more commonly used in the literature relevant to our current topic of discussion. In addition to the freeness of $G_{1, \mu}$ for $\mu\in\{\pm 4,\pm 2i\}$, its discreteness allows us to use our ping-pong argument. Therefore, the following natural question arises: let
\[\mathcal{R}' = \{0\ne\mu\in \C : G_{1,\mu} \text{ is free and discrete}\}.\]
For what values $\mu\in \partial\mathcal{R}'$ does Theorem \ref{complex theorem} hold?
It turns out that $\partial\mathcal{R}'$ is a Jordan curve, and furthermore, $\mathcal{R}'$ is precisely the closure of the well studied \emph{Riley slice} \cite{oshika} (for more information about the Riley slice, see for example \cite{keen-series, komori-series}).
It seems to be reasonable that our results could be extended to other values $\mu \in \partial\mathcal{R}'$, however, some new complications arise.
In the previous section, the specific values of $\mu$ were primarily used in the proofs to select $X_1$ and $X_2$, the corresponding subsets for the ping-pong argument.
However, such subsets may not exist for all $\mu\in \partial\mathcal{R}'$, as the ordinary set $\Omega(G_{1,\mu})$ (i.e., the largest open subset of $\widehat{\C}$ on which $G_{1,\mu}$ acts properly discontinuously) might be empty \cite{maskit}. Nevertheless, the ordinary set is nonempty for a dense collection of $\mu\in \partial\mathcal{R}'$ \cite{cusps}; groups $G_{1,\mu}$ associated with such $\mu$ are called \emph{cusp groups}.
Let us also note that all of the cusp groups with $\mu \in \partial\mathcal{R}'\setminus\{\pm 4, \pm 2i\}$ are \emph{non-classical}, meaning that the Jordan curves defining the boundary of the corresponding subsets $X_1$ and $X_2$ will neither be circles nor lines (see \cite[Theorem 6.18]{gilman-waterman} and also \cite[Figure 3]{gilman}).
\item We return to the normalization $G_{\lambda,2}$ which was used throughout most of the paper.
For every $n\in\N$ define $V_n = \{\lambda\in \C: \sigma(G_{\lambda, 2}) \ge n\}$ and $U_n = V_n\cap\R$.
Theorem \ref{complex theorem} asserts that for every $n$, $V_n$ contains complex open neighborhoods of $\pm 2$ and $\pm i$ and $U_n$ contains real open neighborhoods of $\pm 2$. It would be interesting to study the properties of these sets further; in particular, whether they are open or contain a dense subset of rationals. As previously discussed, many papers have studied groups that admit short relators (syllables-wise), and such groups correspond to values $\lambda\in\C\setminus V_n$ and $\lambda\in\R\setminus U_n$ for small $n$.
Moreover, these values of $\lambda$ admit infinitely many accumulation points \cite{lyndon, beardon, tantan, new-seq}, but this does not rule out the possibility that $\C\setminus V_n$ and $\R\setminus U_n$ are, for example, nowhere dense.
\item It would be interesting to know if our results extend to a larger collection of matrices which converge to $A_2$. For instance, can we replace $A_\lambda$ with some hyperbolic element such as
\[A_\lambda' = \begin{pmatrix}
    1+2\lambda &  2 \\
     \lambda & 1 \end{pmatrix}\]
(with $\lambda>0$)? Of course, this would only be interesting if not all of the groups $\langle A_\lambda', B_2\rangle $ are free.
\item Let $w$ be a relator of $G_{\lambda, \mu}$, where $\mu\in\{2,i\}$. Theorem \ref{complex theorem} tells us that $w$ has a long subword $w'$ of a particular form. It would be interesting to understand whether $w'$ constitutes a significant portion of $w$, specifically by examining the ratios $\syl(w')/\syl(w)$ or $\len(w')/\len(w)$ (recall that the length of a reduced word $w$, denoted $\len(w)$, is its number of letters).
\end{enumerate}

\bigskip

\textbf{Acknowledgements:} The author is grateful to Arie Levit for many helpful ideas and suggestions that contributed to the development of this paper, to Nir Lazarovich for his valuable insight, to Carl-Fredrik Nyberg-Brodda for pointing out relevant literature and to the anonymous referee for their constructive comments.

The author was supported by ISF grant 1788/22.

\bigskip

\DeclareEmphSequence{\itshape}
\bibliographystyle{plain}
\bibliography{references}

\end{document}